\documentclass{amsart}
\usepackage[all]{xy}
\usepackage{color}
\usepackage{amsthm}
\usepackage{amssymb}
\usepackage[colorlinks=true]{hyperref}

\numberwithin{equation}{section}

\newtheorem{theorem}[equation]{Theorem}
\newtheorem*{theorem*}{Theorem}
\newtheorem{lemma}[equation]{Lemma}

\newtheorem*{conjecture*}{Mamma Conjecture}
\newtheorem*{conjecture1*}{Mamma Conjecture (revisited)}
\newtheorem{proposition}[equation]{Proposition}
\newtheorem{corollary}[equation]{Corollary}
\newtheorem*{corollary*}{Corollary}

\theoremstyle{remark}
\newtheorem{definition}[equation]{Definition}

\theoremstyle{remark}

\newcommand{\cA}{{\mathcal A}}
\newcommand{\cB}{{\mathcal B}}
\newcommand{\cC}{{\mathcal C}}
\newcommand{\cD}{{\mathcal D}}

\newcommand{\cJ}{{\mathcal J}}

\newcommand{\cN}{{\mathcal N}}
\newcommand{\cO}{{\mathcal O}}

\newcommand{\cV}{{\mathcal V}}

\newcommand{\bbA}{\mathbb{A}}

\newcommand{\bbL}{\mathbb{L}}

\newcommand{\bbQ}{\mathbb{Q}}
\newcommand{\bbZ}{\mathbb{Z}}

\DeclareMathOperator{\SmProj}{SmProj} 

\DeclareMathOperator{\id}{id}

\DeclareMathOperator{\NChow}{NChow} 
\DeclareMathOperator{\NNum}{NNum} 
\DeclareMathOperator{\Chow}{Chow} 
\DeclareMathOperator{\Num}{Num} 

\newcommand{\dgcat}{\mathsf{dgcat}}

\newcommand{\perf}{\mathsf{perf}}

\newcommand{\dg}{\mathsf{dg}}

\newcommand{\Hom}{\mathsf{Hom}}
\newcommand{\End}{\mathsf{End}}

\newcommand{\rep}{\mathsf{rep}}

\newcommand{\Hmo}{\mathsf{Hmo}}
\newcommand{\op}{\mathsf{op}}

\newcommand{\too}{\longrightarrow}

\newcommand{\ie}{\textsl{i.e.}\ }

\begin{document}

\title[Noncommutative motives, numerical equivalence, and semi-simplicity]{Noncommutative motives, \\numerical equivalence,\\ and semi-simplicity}
\author{Matilde Marcolli and Gon{\c c}alo~Tabuada}

\address{Matilde Marcolli, Mathematics Department, Mail Code 253-37, Caltech, 1200 E.~California Blvd. Pasadena, CA 91125, USA}
\email{matilde@caltech.edu}

\address{Gon\c calo Tabuada, Departamento de Matematica, FCT-UNL, Quinta da Torre, 2829-516 Caparica,~Portugal }
\email{tabuada@fct.unl.pt}

\subjclass[2000]{19A49,19D55,19E15}
\date{\today}

\keywords{Chow and numerical motives, noncommutative motives}

\thanks{The first named author was partially supported by the NSF grants 
DMS-0651925, DMS-0901221 and DMS-1007207. The second named author was partially supported by the FCT-Portugal grant {\tt PTDC/MAT/098317/2008}.}

\begin{abstract}
In this article we further the study of the relationship between pure motives and noncommutative motives, initiated in \cite{CvsNC}. Making use of Hochschild homology, we introduce the category $\NNum(k)_F$ of noncommutative {\em numerical} motives (over a base ring $k$ and with coefficients in a field $F$). We prove that $\NNum(k)_F$ is abelian semi-simple and that Grothendieck's category $\Num(k)_\bbQ$ of numerical motives embeds in $\NNum(k)_\bbQ$ after being factored out by the action of the Tate object. As an application we obtain an alternative proof of Jannsen's semi-simplicity result, which uses the noncommutative world instead of a Weil cohomology.
\end{abstract}

\maketitle
\vskip-\baselineskip
\vskip-\baselineskip
\vskip-\baselineskip

\section{Introduction}
During the last two decades Bondal, Drinfeld, Kaledin, Kapranov, Kontsevich, Van den Bergh, and others, have been promoting a broad noncommutative (algebraic) geometry program; see \cite{Kapranov,BB,Drinfeld,Chitalk,Kaledin,IAS,ENS,Miami,finMot}. This beautiful program, where geometry is performed directly on dg categories (see \S\ref{sec:dg}), encompasses several research fields such as algebraic geometry, representation theory of quivers, symplectic geometry, and even mathematical physics. In analogy with the commutative world, a central problem is the development of an associated theory of {\em noncommutative motives}.
\subsection*{Noncommutative Chow motives}
Let $k$ be a base commutative ring and $F$ a field of coefficients. The category $\NChow(k)_F$ of {\em noncommutative Chow motives} (over $k$ and with coefficients in $F$) was constructed in \cite{IMRN}. It is defined as the pseudo-abelian envelope (see \S\ref{sub:envelope}) of the category whose objects are the smooth and proper dg categories in the sense of Kontsevich (see Definition~\ref{def:sp}), whose morphisms from $\cA$ to $\cB$ are given by the $F$-linearized Grothendieck group $K_0(\cA^\op \otimes_k^\bbL \cB)_F$, and whose composition operation is induced by the (derived) tensor product of bimodules. In analogy with the commutative world, the elements of $K_0(\cA^\op \otimes_k^\bbL \cB)_F$ will be called {\em correspondences}. The category $\NChow(k)_F$ provides a natural framework for the study of several invariants such as cyclic homology (and its variants), algebraic $K$-theory, and even topological Hochschild homology. Among many important applications, $\NChow(k)_F$ allowed a streamlined construction of the Chern characters, a unified and conceptual proof of the fundamental theorem, and even a description of the fundamental isomorphism conjecture in terms of the classical Farrell-Jones isomorphism conjecture; see \cite{BT,CT1, Chern, Duke,Fund}. 

In the particular case where $k$ is a field and $F=\bbQ$, the precise relationship between $\NChow(k)_\bbQ$ and the classical category $\Chow(k)_\bbQ$ of Chow motives was established in \cite{CvsNC}. Recall that $\Chow(k)_\bbQ$ is symmetric monoidal and that it carries an important $\otimes$-invertible object, the Tate motive $\bbQ(1)$. Hence, we can consider the associated orbit category $\Chow(k)_\bbQ\!/_{\!\!-\otimes \bbQ(1)}$; see \S\ref{sub:orbit}. As proved in {\em loc.~cit.}, there is a fully-faithful, $\bbQ$-linear, additive, and symmetric monoidal functor $R$ making the following diagram commutative (up to natural isomorphism)
\begin{equation}\label{diag:1}
\xymatrix@C=2em@R=1.5em{
\SmProj(k)^\op \ar[d]_M \ar@/^2pc/[ddr]^-{NM} & \\
\Chow(k)_\bbQ \ar[d]_\pi & \\
\Chow(k)_\bbQ\!/_{\!\!-\otimes \bbQ(1)} \ar[r]_-R & \NChow(k)_\bbQ\,,
}
\end{equation}
where $NM$ denotes the functor which sends a smooth projective $k$-variety  to its dg category of perfect complexes of $\cO_X$-modules (see \S\ref{sub:smooth}). Intuitively speaking, the above diagram formalizes the conceptual idea that the commutative world can be embedded into the noncommutative world after identifying all the Tate twists.
\subsection*{Motivating questions}
In order to formalize and solve `counting problems', such as counting the number of common points to two planar curves in general position, the category of Chow motives is not appropriate as it makes use of a very refined notion of equivalence. Motivated by these `counting problems', Grothendieck developed in the sixties the category $\Num(k)_F$ of numerical motives; see \cite{GS}. Among many extraordinary properties, Grothendieck conjectured that $\Num(k)_F$ was abelian semi-simple, a result proved thirty years later by Jannsen~\cite{Jannsen}. This circle of conjectures and results lead us to the following important questions: 

\medbreak

{\em Question A:} \textit{Does Grothendieck's category of numerical motives admit a noncommutative analogue ?} 

{\em Question B:} \textit{Does Jannsen's semi-simplicity result hold also in the noncommutative world ?}

{\em Question C:} \textit{Do numerical motives relate to noncommutative numerical motives in the same way as Chow motives relate to noncommutative Chow motives ?}

\medbreak

The purpose of this article is to provide simple and precise answers to these questions. The common answer can be stated in simple terms as: {\em ``Yes, and in the noncommutative world `counting' is expressed in terms of Hochschild homology''}.
\subsection*{Noncommutative numerical motives}

Let $(\cA,e)$ and $(\cB,e')$ be two noncommutative Chow motives and $\underline{X}= (e \circ [\Sigma_i a_i X_i ] \circ e')$ and  $\underline{Y}= (e' \circ [\Sigma_j b_j Y_j ] \circ e)$ two correspondences. Recall that $X_i$ and $Y_j$ are bimodules (see \S\ref{sub:bimodules}), $e$ and $e'$ are idempotent endomorphisms, $a_i$ and $b_j$ are elements of $F$, and that the sums are indexed by a finite set.
\begin{definition}\label{def:intersection}
The {\em intersection number} of $\underline{X}$ with $\underline{Y}$ is given by the formula
\begin{equation}\label{eq:intersection}
\langle \underline{X} \cdot \underline{Y}\rangle := \sum_{i,j} a_i\!\cdot \!b_j \!\cdot \! [HH(\cA; X_i \otimes^\bbL_\cB Y_j)] \in K_0(k)_F\,,
\end{equation}
where $[HH(\cA,X_i \otimes^\bbL_\cB Y_j)]$ denotes the class in $K_0(k)_F$ of the Hochschild homology complex of $\cA$ with coefficients in the $\cA\text{-}\cA$-bimodule $X_i \otimes^\bbL_\cB Y_j$.
\end{definition}
The fact that the above intersection number is well-defined (\ie that it does not depend on the choice of the representatives of $\underline{X}$ and $\underline{Y}$) is explained in the proof of Theorem~\ref{thm:ideal}. Making use of \eqref{eq:intersection}, we can then introduce the following notion of numerical equivalence in the noncommutative world. 
\begin{definition}
A correspondence $\underline{X}$ is {\em numerically equivalent to zero} if for every correspondence $\underline{Y}$ the intersection number $\langle \underline{X} \cdot \underline{Y} \rangle$ is zero.
\end{definition}
\begin{theorem}\label{thm:ideal}
The correspondences which are numerically equivalent to zero form a $\otimes$-ideal of the category $\NChow(k)_F$; which we will denote by $\cN$. 
\end{theorem}
Our candidate solution to the above Question A is then the following:
\begin{definition}
The {\em category of noncommutative numerical motives}, denoted by $\NNum(k)_F$, is the pseudo-abelian envelope of the quotient category $\NChow(k)_F/\cN$.
\end{definition}
Under general hypothesis on the base ring $k$, the ideal $\cN$ (and so the category $\NNum(k)_F$) admits the following conceptual characterization.
\begin{proposition}\label{prop:ideal1}
Assume that $k$ is a local ring (or more generally that $K_0(k)=\bbZ$). Then, the ideal $\cN$ is the largest $\otimes$-ideal of $\NChow(k)_F$ (distinct from the entire category). Moreover, the intersection number~\eqref{eq:intersection} corresponds~to
\begin{equation}\label{eq:intersection1}
\langle \underline{X} \cdot \underline{Y}\rangle = \sum_{i,j,n} (-1)^n\, a_i\!\cdot \!b_j \!\cdot \!\mathrm{rk} HH_n(\cA; X_i \otimes^\bbL_\cB Y_j) \in F\,,
\end{equation}
where $\mathrm{rk} HH_n(\cA; X_i \otimes^\bbL_\cB Y_j)$ denotes the rank of the $n^\mathrm{th}$ Hochschild homology group. 
\end{proposition}
\subsection*{Semi-simplicity}
Although enjoying several important properties, the category $\NChow(k)_F$ is not abelian neither semi-simple. However, when we quotient it by the ideal $\cN$ this situation changes radically. Our solution to the above Question B is then the following.
\begin{theorem}\label{thm:semi-simple}
Assume that one of the following two conditions holds:
\begin{itemize}
\item[(i)] The base ring $k$  is local (or more generally we have $K_0(k)=\bbZ$) and $F$ is a $k$-algebra; a large class of examples is given by taking $k=\bbZ$ and $F$ an arbitrary field.
\item[(ii)] The base ring $k$ is a field extension of $F$; a large class of examples is given by taking $F=\bbQ$ and $k$ a field of characteristic zero.
\end{itemize}
Then, the category $\NNum(k)_F$ is abelian semi-simple. Moreover, if $\cJ$ is a $\otimes$-ideal in $\NChow(k)_F$ for which the pseudo-abelian envelope of the quotient category $\NChow(k)_F/\cJ$ is abelian semi-simple, then $\cJ$ agrees with $\cN$.
\end{theorem}
Roughly speaking, Theorem~\ref{thm:semi-simple} shows that the unique way to obtain an abelian semi-simple category out of $\NChow(k)_F$ is through the use of the above `counting formula' \eqref{eq:intersection1}, defined in terms of Hochschild homology.

The most important result in Grothendieck's theory of pure motives is due to Jannsen~\cite{Jannsen}. It asserts that the category $\Num(k)_\bbQ$ of numerical motives is abelian semi-simple. Making use of the composed functor
\begin{equation}\label{eq:composed}
\Chow(k)_\bbQ \stackrel{\pi}{\too} \Chow(k)_\bbQ\!/_{\!\!-\otimes\bbQ(1)} \stackrel{R}{\too} \NChow(k)_\bbQ \too \NNum(k)_\bbQ\,,
\end{equation}
instead of a Weil cohomology (Jannsen used $l$-adic {\'e}tale cohomology), we obtain an alternative proof of Jannsen's result.
\begin{corollary}\label{cor:Jannsen}
Assume that $k$ is a field of characteristic zero. Then, the category $\Num(k)_\bbQ$ is abelian semi-simple.
\end{corollary}

\subsection*{Relationship with numerical motives}
The relationship between Chow motives and noncommutative Chow motives described in diagram \eqref{diag:1} admits a numerical analogue. Our solution to the above Question C, which emphasizes the correctness of our construction, is the following.
\begin{theorem}\label{thm:embedding}
There exists a fully-faithful, $\bbQ$-linear, additive, and symmetric monoidal functor $R_{\cN}$ making the following diagram commutative (up to natural isomorphism)
$$
\xymatrix@C=2em@R=1.5em{
& \Chow(k)_\bbQ  \ar[d]^\pi \ar[dl] &&  \\
\Num(k)_\bbQ \ar[d]_\pi & \Chow(k)_\bbQ/_{\!\!-\otimes \bbQ(1)}  \ar[dl]  \ar[rr]^-R && \ar[dl] \NChow(k)_\bbQ  \\
\Num(k)_\bbQ/_{\!\!-\otimes \bbQ(1)} \ar[rr]_{R_{\cN}} & & \NNum(k)_\bbQ & \,.
}
$$
\end{theorem}
Intuitively speaking, Theorem~\ref{thm:embedding} formalizes the conceptual idea that Hochschild homology is the correct way to express `counting' in the noncommutative world. 

We believe the above `bridge' between the commutative and noncommutative worlds will open new horizons and opportunities of research by enabling the interchange of results, techniques, ideas, and insights between the two worlds. This is illustrated in the following two corollaries.
\begin{corollary}\label{cor:applic1}
Let $f$ be a morphism (\ie a correspondence) in $\Chow(k)_\bbQ$, resp. in $\Num(k)_\bbQ$, between any two motives. Then, $f$ is an isomorphism if and only if $(R\circ \pi)(f)$, resp. $(R_\cN\circ \pi)(f)$, is an isomorphism.
\end{corollary}
Informally speaking, Corollary~\ref{cor:applic1} shows that if by hypothesis two non isomorphic motives become isomorphic in the noncommutative world, then the isomorphism in question must be `purely' noncommutative, \ie it {\em cannot} be induced by a correspondence.
\begin{corollary}\label{cor:applic2}
Assume that $k$ is a field of characteristic zero. Then, for every numerical motive $(X,p,m)$ the $\bbQ$-algebra of endomorphisms $\End_{\Num(k)_\bbQ}((X,p,m))$ is finite dimensional.
\end{corollary}
Corollary~\ref{cor:applic2} was first proved by Kleiman in \cite{Kleiman} using and appropriate Weil cohomology towards a gentle category of graded vector spaces. In contrast, our proof doesn't make use of any Weil cohomology. It is simply based on the finite-dimensionality of the noncommutative world and on the above `bridge'. 

\section{Differential graded categories}\label{sec:dg}
In this section we collect the notions and results concerning dg categories which are used throughout the article. For further details we invite the reader to consult Keller's ICM adress~\cite{ICM}. Let $k$ be a fixed base commutative ring and $\cC(k)$ the category of (unbounded) cochain complexes of $k$-modules. A {\em differential graded (=dg) category $\cA$} is a category enriched over $\cC(k)$. Concretely, the morphisms sets $\cA(x,y)$ are complexes of $k$-modules and the composition operation fulfills the Leibniz rule: $d(f\circ g)=(df)\circ g+(-1)^{\textrm{deg}(f)}f\circ(dg)$. A {\em dg functor $F:\cA \to \cB$} is simply a functor which preserves the differential graded structure. The category of dg categories will be denoted by $\dgcat(k)$.

\subsection{Dg modules}\label{sub:modules}
Let $\cA$ be a dg category. Its {\em opposite} dg category $\cA^\op$ has the same objects and complexes of morphisms given by $\cA^\op(x,y):=\cA(y,x)$. A {\em right dg $\cA$-module} (or simply a $\cA$-module) is a dg functor $\cA^\op \to \cC_\dg(k)$ with values in the dg category of complexes of $k$-modules. We will denote by $\cC(\cA)$ the category of $\cA$-modules and by $\cD(\cA)$ the {\em derived category} of $\cA$, \ie the localization of $\cC(\cA)$ with respect to the class of quasi-isomorphisms; see \cite[\S3]{ICM}. The full triangulated subcategory of compact objects (see \cite[Def.~4.2.7]{Neeman}) will be denoted by $\cD_c(\cA)$.
\subsection{Morita model structure}\label{sub:Morita}
As proved in \cite{IMRN}, the category $\dgcat(k)$ carries a Quillen model structure whose weak equivalence are the {\em derived Morita equivalences}, \ie those dg functors $F:\cA \to \cB$ which induce an equivalence $\cD(\cA) \stackrel{\sim}{\to} \cD(\cB)$ on the associated derived categories. The homotopy category hence obtained will be denoted by $\Hmo(k)$. 
\subsection{Symmetric monoidal structure}\label{sub:monoidal}
The tensor product of $k$-algebras extends naturally to dg categories, giving rise to a symmetric monoidal structure $-\otimes_k-$ on $\dgcat(k)$. The $\otimes$-unit is the dg category $\underline{k}$ with one object and with $k$ as the dg algebra of endomorphisms (concentrated in degree zero). As explained in \cite[\S4.3]{ICM}, the tensor product of dg categories can be derived giving rise to a symmetric monoidal category $(\Hmo(k), -\otimes_k^\bbL-, \underline{k})$. 
\subsection{Bimodules and Hom-spaces}\label{sub:bimodules}
Let $\cA$ and $\cB$ be two dg categories. A {\em $\cA\text{-}\cB$-bimodule} is a dg functor $\cA\otimes^\bbL_k \cB^\op \to \cC_\dg(k)$, or in other words a $(\cA^\op \otimes_k^\bbL \cB)$-module. Let $\rep(\cA,\cB)$ be the full triangulated subcategory of $\cD(\cA^\op \otimes_k^\bbL \cB)$ spanned by the (cofibrant) $\cA\text{-}\cB$-bimodules $X$ such that for every object $x \in \cA$ the associated $\cB$-module $X(x,-)$ belongs to $\cD_c(\cB)$. As explained in \cite[\S4.2 and \S4.6]{ICM} there is a natural bijection $\Hom_{\Hmo(k)}(\cA,\cB)\simeq \mathrm{Iso}\,\rep(\cA,\cB)$, where $\mathrm{Iso}\, \rep(\cA,\cB)$ stands for the isomorphism classes of objects in $\rep(\cA,\cB)$. Under this identification, the composition operation in $\Hmo(k)$ corresponds to the (derived) tensor product of bimodules.

\subsection{Smooth and proper dg categories}\label{sub:smooth}
In his noncommutative (algebraic) geometry program \cite{IAS,ENS,Miami,finMot}, Kontsevich introduced the following important notions of smoothness and properness.
\begin{definition}[Kontsevich]\label{def:sp}
Let $\cA$ be a dg category. We say that $\cA$ is {\em smooth} if the $\cA\text{-}\cA$-bimodule
\begin{eqnarray}\label{eq:bimodule}
\cA(-,-): \cA \otimes_k^\bbL \cA^\op \to \cC_\dg(k) && (x,y) \mapsto \cA(x,y)
\end{eqnarray}
belongs to $\cD_c(\cA^\op \otimes_k^\bbL \cA)$. We say that $\cA$ is {\em proper} if for each ordered pair of objects $(x,y)$ in $\cA$, the complex $\cA(x,y)$ of $k$-modules belongs to $\cD_c(k)$.
\end{definition}
Let $X$ be a smooth and proper $k$-scheme. Then, the associated dg category $\cD_\perf^\dg(X)$ of perfect complexes of $\cO_X$-modules is smooth and proper in the sense of Definition~\ref{def:sp}; see \cite[Example~4.5]{CT1}. For further examples of smooth and proper dg categories, coming from representation theory of quivers and from deformation by quantization, we invite the reader to consult \cite{IAS}. 

As proved in \cite[Thm.~4.8]{CT1}, the smooth and proper dg categories can be characterized conceptually as being precisely the dualizable (or rigid) objects of the symmetric monoidal category $(\Hmo(k),-\otimes_k^\bbL-,\underline{k})$. As a consequence, we have a natural bijection $\Hom_{\Hmo(k)}(\cA,\cB)\simeq \mathrm{Iso}\, \cD_c(\cA^\op \otimes^\bbL_k \cB)$ whenever $\cA$ and $\cB$ are smooth and proper.

\section{Orbits, quotients, and pseudo-abelian envelopes}
In this section we prove that orbits, quotients, and pseudo-abelian envelopes are three distinct operations which, under very general hypotheses, commute with each other; see Proposition~\ref{prop:commutativity}. This general result, which is of independent interest, will play a key role in the proof of Theorem~\ref{thm:embedding}. In what follows, $F$ will denote a fixed field and $(\cC,\otimes, {\bf 1})$ a $F$-linear, additive and rigid symmetric monoidal category.
\subsection{Orbit categories}\label{sub:orbit}
Let $\cO$ be a $\otimes$-invertible object of $\cC$. Then, as explained in \cite[\S7]{CvsNC}, we can construct the {\em orbit category $\cC\!/_{\!\!-\otimes\cO}$ of $\cC$} as follows: the objects are the same as those of $\cC$ and the morphisms are given by 
\begin{equation*}
\Hom_{\cC\!/_{\!\!-\otimes \cO}}(X,Y):=\bigoplus_{j \in \bbZ} \Hom_\cC(X,Y \otimes \cO^{\otimes j})\,.
\end{equation*}
Composition is naturally induced by $\cC$. The category $\cC\!/_{\!\!-\otimes\cO}$ is $F$-linear, additive, rigid symmetric monoidal, and comes equipped with a natural $F$-linear, additive, faithful, (essentially) surjective, and symmetric monoidal functor $\pi:\cC \to \cC\!/_{\!\!-\otimes\cO}$.
\subsection{Quotient categories}\label{sub:quotient}
As proved in \cite[Lemma~7.1.1]{AK}, the formula
\begin{equation}\label{eq:description}
 \cN(X,Y):= \{ f \in \Hom_\cC(X,Y) \, |\, \forall \, g \in \Hom_\cC(Y,X),\, \mathrm{tr}(g \circ f)=0\}
\end{equation}
defines a $\otimes$-ideal $\cN$ of $\cC$. Here, $\mathrm{tr}(g \circ f)$ denotes the {\em categorical trace} in the rigid symmetric monoidal category $\cC$, \ie the composed morphism
\begin{equation*}
\left({\bf 1} \stackrel{(g\circ f)^\sharp}{\too} X^\vee \otimes X \stackrel{\sim}{\too} X \otimes X^\vee \stackrel{\mathrm{ev}}{\too} {\bf 1}\right) \in \mathsf{End}_\cC({\bf 1})\,,
\end{equation*}
where $(g \circ f)^\sharp$ is the morphism naturally associated to $(g \circ f)$. By taking the quotient of $\cC$ by the $\otimes$-ideal $\cN$ we obtain then a $F$-linear, additive and rigid symmetric monoidal category $\cC/\cN$, as well as a natural $F$-linear, additive, full, (essentially) surjective and symmetric monoidal functor $\cC \to \cC/\cN$.

In the particular case where $\End_\cC({\bf 1}) \simeq F$, the ideal $\cN$ can be characterized as the largest $\otimes$-ideal of $\cC$ (distinct from the entire category); see \cite[Prop.~7.1.4]{AK}.

\subsection{Pseudo-abelian envelope}\label{sub:envelope}
Given a category $(\cC, \otimes, {\bf 1})$ as above, we can construct its pseudo-abelian envelope $\cC^\natural$ as follows: the objects are the pairs $(X,e)$, where $X \in \cC$ and $e$ is an idempotent of the $F$-algebra $\End_\cC(X)$, and the morphisms are given by 
$$ \Hom_{\cC^\natural}((X,e),(Y,e')):=e \circ \Hom_\cC(X,Y) \circ e'\,.$$
Composition is naturally induced by $\cC$. The symmetric monoidal structure on $\cC$ extends naturally to $\cC^\natural$ by the formula $(X,e)\otimes(Y,e'):=(X \otimes Y, e\otimes e')$. We obtain then a $F$-linear, additive, idempotent complete, rigid symmetric monoidal category $\cC^\natural$, as well as a $F$-linear, additive, fully-faithful and symmetric monoidal functor $\cC \to \cC^\natural$, $X \mapsto (X, \id_X)$.

\begin{proposition}\label{prop:commutativity}
Let $(\cD,\otimes, {\bf 1})$ be a $F$-linear, additive, rigid symmetric monoidal category, and $\cO$ a $\otimes$-invertible object of $\cD$. Assume that 
\begin{equation}\label{eq:hyp}
\Hom_\cD({\bf 1}, \cO^{\otimes n}) \simeq  \left\{ \begin{array}{ll} F  &  n=0 \\ 0 & n \neq 0 \,.\end{array} \right.
\end{equation}
Then, there exist $F$-linear, additive and symmetric monoidal functors $\alpha$, $\beta$ and $\gamma$ making the following diagram commutative
\begin{equation*}
\xymatrix{
\cD \ar[d]_\pi \ar[r] & \cD/\cN \ar[d]^\pi \ar[r] & (\cD/\cN)^\natural \ar[d]^\pi \\
\cD\!/_{\!\!-\otimes \cO} \ar[d] \ar[r]^-\alpha & (\cD/\cN)\!/_{\!\!-\otimes \cO} \ar@{=}[d] \ar[r]^-\gamma & (\cD/\cN)^\natural \!/_{\!\!-\otimes \cO} \ar@{=}[d] \\
(\cD\!/_{\!\!-\otimes \cO})/\cN \ar[d] \ar[r]^-\beta_-\simeq & (\cD/\cN)\!/_{\!\!-\otimes \cO} \ar[d] \ar[r]^-\gamma & (\cD/\cN)^\natural\!/_{\!\!-\otimes \cO} \ar[d] \\
((\cD\!/_{\!\!-\otimes \cO})/\cN)^\natural \ar[r]^-{\beta^\natural}_-\simeq & ((\cD/\cN)\!/_{\!\!-\otimes \cO})^\natural \ar[r]^-{\gamma^\natural}_-\simeq & ((\cD/\cN)^\natural\!/_{\!\!-\otimes \cO})^\natural \,.
}
\end{equation*}
\end{proposition}
\begin{proof}
The functor $\alpha$ is defined as being the identity map on objects and the following natural surjection 
$$ \bigoplus_{j \in \bbZ} \Hom_{\cD}(X, Y \otimes \cO^{\otimes j}) \twoheadrightarrow \bigoplus_{j \in \bbZ} \Hom_{\cD/\cN}(X, Y \otimes \cO^{\otimes j})$$
on morphisms. The fact that it is $F$-linear, additive and symmetric monoidal follows from the corresponding properties of the natural functor $\cD \to \cD/\cN$.

In order to define the functor $\beta$ we need then to show that $\alpha$ vanishes on the $\otimes$-ideal $\cN$. Let
$$\underline{f}=\{f_j\}_{j \in \bbZ} \in \bigoplus_{j \in \bbZ} \Hom_{\cD}(X, Y \otimes \cO^{\otimes j})$$
be a morphism in the orbit category $\cD\!/_{\!\!-\otimes \cO}$ belonging $\cN(X,Y)$. We need to show that each one of the morphisms $f_j$ in $\cD$ belongs to $\cN(X,Y \otimes \cO^{\otimes j})$. By the definition \eqref{eq:description}
of the $\otimes$-ideal $\cN$, we have $\mathrm{tr}(\underline{g} \circ \underline{f})=0$ for every morphism 
$$\underline{g} =\{g_i\}_{i \in \bbZ} \in \bigoplus_{i \in \bbZ} \Hom_{\cD}(Y, X \otimes \cO^{\otimes i})$$
in the orbit category $\cD\!/_{\!\!-\otimes \cO}$. Therefore, let us consider the family of morphisms $\underline{g^j_h} \in \Hom_{\cD\!/_{\!\!-\otimes \cO}}(Y,X)$, where $h \in \Hom_\cD(Y \otimes \cO^{\otimes j},X)$, $j$ is an integer, and
\begin{equation*}
(g_h^j)_i :=  \left\{ \begin{array}{ll} h \otimes \cO^{\otimes(-j)} &  i=-j \\ 0 & i \neq -j \,.\end{array} \right.
\end{equation*}
Recall from \cite[\S7]{CvsNC} that the (co-)evaluation maps of the symmetric monoidal structure on $\cD\!/_{\!\!-\otimes \cO}$ are the image of those of $\cD$ by the natural functor $\pi:\cD \to \cD\!/_{\!\!-\otimes \cO}$. Hence, a direct inspection of the categorical trace $\mathrm{tr}(\underline{g^j_h} \circ \underline{f}) \in \End_{\cD\!/_{\!\!-\otimes \cO}}({\bf 1}, {\bf 1})$ show us that its zero component is precisely the categorical trace $\mathrm{tr}(h \circ f_j) \in \End_\cD({\bf 1}, {\bf 1})$. By considering an arbitrary morphism $h$ and an arbitrary integer $j$, we conclude then that each one of the morphisms $f_j$ in $\cD$ belongs to $\cN(X,Y \otimes \cO^{\otimes j})$. As a consequence, the functor $\alpha$ factors through $(\cD\!/_{\!\!-\otimes \cO})/\cN$ and thus gives rise to a $F$-linear, additive and symmetric monoidal functor $\beta$ as in the above diagram.

Let us now prove that $\beta$ is an equivalence of categories. Since the functor $\alpha$ is full and (essentially) surjective, so it is the functor $\beta$. It remains then to show that $\beta$ is moreover faithful. In order to show this let us consider the kernel $\mathrm{Ker}(\alpha)$ of the functor $\alpha$. As $\alpha$ is symmetric monoidal, $\mathrm{Ker}(\alpha)$ is a $\otimes$-ideal of $\cD\!/_{\!\!-\otimes \cO}$. Note that the above assumption \eqref{eq:hyp} implies that $\End_{\cD\!/_{\!\!-\otimes \cO}}({\bf 1})\simeq F$. Hence, as explained in \S\ref{sub:quotient}, the ideal $\cN$ is the largest $\otimes$-ideal of $\cD\!/_{\!\!-\otimes \cO}$. As a consequence $\mathrm{Ker}(\alpha)\subseteq \cN$. We obtain then the following commutative diagram
\begin{equation*}
\xymatrix{
\cD\!/_{\!\!-\otimes \cO} \ar[r]^-{\alpha} \ar[d] & (\cD/\cN)\!/_{\!\!-\otimes \cO} \ar@{=}[dd] \\
(\cD\!/_{\!\!-\otimes \cO})/\mathrm{Ker}(\alpha) \ar[d]_\delta \ar[dr]^{\simeq} & \\
(\cD\!/_{\!\!-\otimes\cO})/\cN \ar[r]_-{\beta} & (\cD/\cN)\!/_{\!\!-\otimes \cO}\,,
}
\end{equation*}
where the ``diagonal'' functor is an equivalence of categories since $\alpha$ is (essentially) surjective. Finally, since the induced functor $\delta$ is full and (essentially) surjective, we conclude from the above commutative diagram that $\beta$ is faithful.

Let us now focus on the functor $\gamma$, induced by the natural functor $\cD/\cN \to (\cD/\cN)^\natural$. Clearly, it is $F$-linear, additive and symmetric monoidal. Recall that by construction, every object in $(\cD/\cN)^\natural$ is a direct factor of an object in the image of the natural functor $\cD/\cN \to (\cD/\cN)^\natural$. Hence, this property holds also for $\gamma$. By passing to its pseudo-abelian envelope $\gamma^\natural$ we obtain then an equivalence of categories.
\end{proof}

\section{Proof of Theorem~\ref{thm:ideal}}
By construction, the category $\NChow(k)_F$ is $F$-linear, additive and rigid symmetric monoidal. Hence, following \S\ref{sub:quotient}, the proof will consist of showing that the intersection number \eqref{eq:intersection} agrees with the categorical trace of the correspondence $\underline{Y} \circ \underline{X}$; see Corollary~\ref{cor:intersection}. Note that this equality implies automatically that the intersection number \eqref{eq:intersection} is well-defined, \ie that it does not depend on the choice of the representatives of $\underline{X}$ and $\underline{Y}$. Let us then focus on the computation of the categorical trace $\mathrm{tr}(\underline{Y} \circ \underline{X}) \in K_0(k)_F$. Recall that the correspondence 
\begin{equation}\label{eq:correspondance} 
\underline{Y} \circ \underline{X} = e \circ [ \sum_{i,j} (a_i \!\cdot\! b_j)(X_i \otimes^\bbL_\cB Y_j)] \circ e
\end{equation}
is an endomorphism of the noncommutative Chow motive $(\cA,e)$. By construction of $\NChow(k)_F$, we observe that $\mathrm{tr}(\underline{Y} \circ \underline{X})$ agrees with the categorical trace of the correspondence
\begin{equation}\label{eq:correspondance1} 
\underline{Z}:= [ \sum_{i,j} (a_i \!\cdot\! b_j)(X_i \otimes^\bbL_\cB Y_j)] \in \End_{\NChow(k)_F}((\cA,\id_\cA))\,.
\end{equation}
\begin{proposition}\label{prop:aux}
Let $\cA$ be a smooth and proper dg category in the sense of Kontsevich (see \S\ref{sub:smooth}) and $W$ an $\cA\text{-}\cA$-bimodule which belongs to $\cD_c(\cA^\op \otimes^\bbL \cA)$. Then, the categorical trace of the correspondence $[W] \in \End_{\NChow(k)_F}((\cA,\id_\cA))$ is given~by
\begin{equation*}
\mathrm{tr}([W])=[HH(\cA; W)] \in K_0(k)_F\,,
\end{equation*}
where $[HH(\cA; W)]$ denotes the class in $K_0(k)_F$ of the Hochschild homology complex of $\cA$ with coefficients in the $\cA\text{-}\cA$-bimodule $W$.
\end{proposition}
\begin{proof}
As explained in \S\ref{sub:smooth}, the conditions on $\cA$ and $W$ imply that $\cA$ is a dualizable object of the symmetric monoidal category $(\Hmo(k), -\otimes^\bbL_k-, \underline{k})$ and that the $\cA\text{-}\cA$-bimodule $W$ is an endomorphism of $\cA$. We start then by showing that the categorical trace of $W$ in $\Hmo(k)$ agrees with the isomorphism class in $\cD_c(\underline{k}) \simeq \cD_c(k)$ of the Hochschild homology complex $HH(\cA;W)$. Recall from \cite[\S4]{CT1} that the dual of $\cA$ is its opposite dg category $\cA^\op$ and the evaluation map $\cA\otimes_k^\bbL\cA^\op \stackrel{\mathrm{ev}}{\to} \underline{k}$ is given by the $\cA\text{-}\cA$-bimodule \eqref{eq:bimodule}. Hence, as explained in \S\ref{sub:quotient}, the categorical trace of $W$ corresponds to the isomorphism class in $\cD_c(k)$ of the following composition of bimodules
$$\underline{k} \stackrel{W}{\too} \cA^\op \otimes_k^\bbL\cA \stackrel{\sim}{\too} \cA\otimes_k^\bbL\cA^\op \stackrel{\cA(-,-)}{\too} \underline{k} \,.$$ 
By performing this composition we obtain then the complex  
$$W \otimes^\bbL_{\cA^\op \otimes^\bbL_k \cA}\cA(-,-)\,,$$ 
which as explained in \cite[Prop.~1.1.13]{Loday} is quasi-isomorphic to the Hochschild homology complex $HH(\cA;W)$ of $\cA$ with coefficients in the $\cA\text{-}\cA$-bimodule $W$. This proves our claim. 

Now, let us denote by $\Hmo(k)^{\mathsf{sp}} \subset \Hmo(k)$ the full subcategory of smooth and proper dg categories in the sense of Kontsevich. Note that we have a natural functor
\begin{eqnarray*}
\Hmo(k)^{\mathsf{sp}} \too \NChow(k)_F
\end{eqnarray*}
which maps $\cA$ to $(\cA, \id_\cA)$ and sends a dg functor $F: \cA \to \cB$ to the class in $K_0(\cA^\op \otimes_k^\bbL \cB)_F$ of the $\cA\text{-}\cB$-bimodule $(x,y) \mapsto \cB(y, F(x))$. By construction, this functor is symmetric monoidal and so it preserves the categorical trace. When restricted to the endomorphisms of the $\otimes$-unit objects, it corresponds to the map  
$$ \mathrm{Iso}\, \cD_c(k) \simeq \End_{\Hmo(k)^{\mathsf{sp}}}(\underline{k}) \too \End_{\NChow(k)_F}((\underline{k}, \id_{\underline{k}})) \simeq K_0(k)_F $$
which sends an isomorphism class in $\cD_c(k)$ to the respective class in the $F$-linearized Grothendieck group $K_0(\cD_c(k))_F\simeq K_0(k)_F$. Hence, we conclude that $\mathrm{tr}([W])=[HH(\cA;W)] \in K_0(k)_F$ and so the proof is finished.
\end{proof}
The trace formula for the intersection number given by Kleiman in \cite{Kleiman} admits the following noncommutative analogue.

\begin{corollary}\label{cor:intersection}
The intersection number \eqref{eq:intersection} agrees with the categorical trace of the
correspondence $\underline{Y} \circ \underline{X}$.
\end{corollary}

\begin{proof}
As in any $F$-linear rigid symmetric monoidal category, the categorical trace gives rise to a $F$-linear homomorphism
$$\mathrm{tr}(-): \End_{\NChow(k)_F}((\cA,\id_\cA)) \too \End_{\NChow(k)_F}((\underline{k}, \id_{\underline{k}}))\simeq K_0(k)_F\,.$$
Therefore, the categorical trace of the correspondance \eqref{eq:correspondance1}, which agrees with the categorical trace of \eqref{eq:correspondance}, identifies with 
$$ \sum_{i,j} a_i \!\cdot\! b_j \!\cdot \!\mathrm{tr}([X_i \otimes^\bbL_\cB Y_j])\,.$$
By applying Proposition~\ref{prop:aux}, with $W=X_i \otimes^\bbL_\cB Y_j$, we obtain then the desired equality
$$\mathrm{tr}(\underline{Z}) = \sum_{i,j} a_i \!\cdot \!b_j\! \cdot \! [HH(\cA;X_i \otimes_\cB^\bbL Y_j)] =: \langle \underline{X} \cdot \underline{Y} \rangle\,.$$
\end{proof}
\section{Proof of Proposition~\ref{prop:ideal1}}
When $k$ is a local ring, we have a natural isomorphism $K_0(k)\simeq \bbZ$; see \cite[Chap.~1, \S3]{Rosenberg}. This implies that
$$ \End_{\NChow(k)_F}(\underline{k})=K_0(k)_F \simeq F\,.$$
Therefore, since $\NChow(k)_F$ is $F$-linear, additive and rigid symmetric monoidal, we conclude from \S\ref{sub:quotient} that the ideal $\cN$ can be characterized as the largest $\otimes$-ideal of $\NChow(k)_F$ (distinct from the entire category). The fact that the intersection number \eqref{eq:intersection} corresponds to \eqref{eq:intersection1} follow from the natural identification
\begin{eqnarray*}
K_0(\cD_c(k)) \simeq K_0(k) \stackrel{\sim}{\to} \bbZ && [M] \mapsto \sum_n (-1)^n \mathrm{rk}H^n(M)\,, 
\end{eqnarray*}
where $\mathrm{rk}H^n(M)$ denotes the rank of the $n^{\mathrm{th}}$ cohomology group of the complex of $k$-modules $M$.

\section{Proof of Theorem~\ref{thm:semi-simple}}
The proof will consist of verifying the conditions of Andr{\'e}-Kahn's general result \cite[Thm.~1]{AK-errata}. Our $F$-linear rigid symmetric monoidal category will be the category $\NChow(k)_F$ of noncommutative Chow motives; in \cite{AK-errata} the authors used the letter $K$ instead of $F$. Note that since by hypothesis $k$ is a local ring (or even a field), we have $K_0(k)=\bbZ$. This implies that 
$$\End_{\NChow(k)_F}(\underline{k})=K_0(k)_F\simeq F\,.$$
The proof will consist then on constructing a $F$-linear functor
\begin{equation}\label{eq:searched}
\NChow(k)_F \too \cV
\end{equation}
towards a $L$-linear rigid symmetric monoidal category (where $L$ is a field extension of $K$) which satisfies the following conditions:
\begin{itemize}
\item[(a)] the $L$-linear Hom-spaces in $\cV$ are finite dimensional and 
\item[(b)] the nilpotent endomorphisms in $\cV$ have a trivial categorical trace.
\end{itemize}
Recall from \cite[\S5]{IMRN} the construction of the additive category $\Hmo_0(k)$: the objects are the dg categories, the morphisms from $\cA$ to $\cB$ are given by the Grothendieck group $K_0 \rep(\cA,\cB)$ (see \S\ref{sub:bimodules}), and the composition operation is induced by the (derived) tensor product of bimodules. Moreover, the symmetric monoidal structure on $\Hmo(k)$ extends naturally to $\Hmo_0(k)$. There is a natural symmetric monoidal functor $\Hmo(k) \to \Hmo_0(k)$ that is the identity on objects and which sends an element $X$ of $\rep(\cA,\cB)$ to the corresponding class $[X]$ in the Grothendieck group $K_0\rep(\cA,\cB)$. As consequence, the smooth and proper dg categories in the sense of Kontsevich (see \S\ref{sub:smooth}) are still dualizable objects in the symmetric monoidal category $\Hmo_0(k)$; see \S\ref{sub:smooth}. Let us denote by $\Hmo_0(k)^{\mathsf{sp}}\subset \Hmo_0(k)$ the full subcategory of smooth and proper dg categories. Note that $\NChow(k)_F$ is obtained from $\Hmo_0(k)^{\mathsf{sp}}$ by first tensoring each abelian group of morphisms with the field $F$ and then passing to the associated pseudo-abelian envelope. Schematically, we have the following composition
$$ \Hmo_0(k)^{\mathsf{sp}} \stackrel{(-)_F}{\too} \Hmo_0(k)_F^{\mathsf{sp}} \stackrel{(-)^\natural}{\too} \NChow(k)_F\,.$$
Therefore, in order to construct a symmetric monoidal $F$-linear functor as above \eqref{eq:searched}, it suffices then to construct a symmetric monoidal functor
\begin{equation}\label{eq:functor2}
\Hmo_0(k)^{\mathsf{sp}} \too \cV
\end{equation}
towards an idempotent complete $L$-linear category $\cV$. Recall from \cite[\S6.1]{IMRN} that Hochschild homology ($HH$) gives rise to a symmetric monoidal functor $\overline{HH}: \Hmo_0(k) \to \cD(k)$. Since $\overline{HH}$ is symmetric monoidal, it maps dualizable objects to dualizable objects, and so it restricts to a symmetric monoidal functor
\begin{equation}\label{eq:HH}
\overline{HH}: \Hmo_0(k)^{\mathsf{sp}} \too \cD_c(k)\,.
\end{equation}
Now, let us assume that condition (i) holds. Since $F$ is a $k$-algebra we have a ring homomorphism $k \to F$ and so we can consider the associated (derived) extension of scalars functor
\begin{equation}\label{eq:extension}
\cD_c(k) \too \cD_c(F)\,.
\end{equation} 
Note that this functor is symmetric monoidal and that $\cD_c(F)$ is an idempotent complete $F$-linear category. The category $\cD_c(F)$ is in fact naturally equivalent to the category of $\bbZ$-graded $F$-vector spaces. Hence, we can take for \eqref{eq:functor2} the composition of the functors \eqref{eq:HH} and \eqref{eq:extension}, where $L=F$. The fact that $\cD_c(F)$ is rigid symmetric monoidal and that it satisfies the above conditions (a) and (b) is clear.

Now, let us assume that condition (ii) holds. Note that the category $\cD_c(k)$ is $k$-linear. Since $k$ is a $F$-algebra, the category $\cD_c(k)$ is moreover $F$-linear. Hence, since $\cD_c(k)$ is idempotent complete, we can take for \eqref{eq:functor2} the functor \eqref{eq:HH}, where now $L=K$. The fact that $\cD_c(k)$ rigid symmetric monoidal and that it satisfies the above conditions (a) and (b) is clear.

\section{Proof of Corollary~\ref{cor:Jannsen}}
Since the category $\Num(k)_\bbQ$ of numerical motives identifies with the pseudo-abelian envelope of the quotient category $\Chow(k)_\bbQ/\cN$ (see \cite[Examples 6.3.1 and 7.1.2]{AK}), the proof will consist (as the proof of Theorem~\ref{thm:semi-simple}) on verifying the conditions of Andr{\'e}-Kahn's general result \cite[Thm.~1]{AK-errata}. Our $\bbQ$-linear rigid symmetric monoidal category is the category $\Chow(k)_\bbQ$ of Chow motives and our $\bbQ$-linear symmetric monoidal functor is the functor \eqref{eq:composed}. Since by Theorem~\ref{thm:semi-simple} the category $\NNum(k)_\bbQ$ is abelian semi-simple, nilpotent endomorphisms in $\NNum(k)_\bbQ$ have a trivial categorical trace. Moreover, as explained in the proof of Theorem~\ref{thm:semi-simple}, the $\bbQ$-linear Hom-spaces of $\NNum(k)_\bbQ$ are finite dimensional. Hence, all the conditions are satisfied and so the proof is finished.

\section{Proof of Theorem~\ref{thm:embedding}}
Recall from diagram \eqref{diag:1} that the functor $R$ is symmetric monoidal. Hence, it maps the $\otimes$-ideal $\cN$ of $\Chow(k)_\bbQ\!/_{\!\!-\otimes \bbQ(1)}$ to the $\otimes$-ideal $\cN$ of $\NChow(k)_\bbQ$, thus giving rise to the following commutative diagram
\begin{equation}\label{eq:aux1}
\xymatrix{
\Chow(k)_\bbQ\!/_{\!\!-\otimes \bbQ(1)} \ar[d] \ar[rr]^-R && \NChow(k)_\bbQ \ar[d] \\
(\Chow(k)_\bbQ\!/_{\!\!-\otimes \bbQ(1)})/\cN \ar[rr]_-{R/\cN} \ar[d] && \NChow(k)_\bbQ/\cN \ar[d] \\
\left( (\Chow(k)_\bbQ\!/_{\!\!-\otimes \bbQ(1)})/\cN \right)^\natural \ar[rr]_-{(R/\cN)^\natural} && (\NChow(k)_\bbQ/\cN)^\natural = \NNum(k)_\bbQ\,.
}
\end{equation}
Now, recall that we have the following computation in the category of Chow motives
\begin{equation}\label{eq:computation}
\Hom_{\Chow(k)_\bbQ}(\mathrm{Spec}(k),\bbQ(n)) \simeq  \left\{ \begin{array}{ll} \bbQ  &  n=0 \\ 0 & n \neq 0 \,.\end{array} \right.
\end{equation}
By the construction of the orbit category we obtain then the natural isomorphism
$$\End_{\Chow(k)_\bbQ\!/_{\!\!-\otimes \bbQ(1)}}(\mathrm{Spec}(k))\simeq \bbQ\,.$$
Hence, if in Proposition~\ref{prop:commutativity} we take $\cD=\Chow(k)_\bbQ$ and $\cO=\bbQ(1)$, then all the conditions are satisfied. As a consequence, we obtain the commutative diagram
\begin{equation*}
\xymatrix@C=1.5em@R=2.5em{
\Chow(k)_\bbQ \ar[d]_\pi \ar[r] & \Chow(k)_\bbQ / \cN \ar[d]^\pi \ar[r] & (\Chow(k)_\bbQ / \cN)^\natural \simeq \Num(k)_\bbQ \ar[d]^\pi \\
\Chow(k)_\bbQ\!/_{\!\!-\otimes \bbQ(1)} \ar[d] \ar[r]^-\alpha & (\Chow(k)_\bbQ/\cN)\!/_{\!\!-\otimes \bbQ(1)} \ar@{=}[d] \ar[r]^-\gamma & \Num(k)_\bbQ \!/_{\!\!-\otimes \bbQ(1)} \ar@{=}[d] \\
(\Chow(k)_\bbQ\!/_{\!\!-\otimes \bbQ(1)})/\cN \ar[d] \ar[r]^-\beta_-\simeq & (\Chow(k)_\bbQ/\cN)\!/_{\!\!-\otimes \bbQ(1)} \ar[d] \ar[r]^-\gamma & \Num(k)_\bbQ\!/_{\!\!-\otimes \bbQ(1)} \ar[d]  \\
((\Chow(k)_\bbQ\!/_{\!\!-\otimes \bbQ(1)})/\cN)^\natural \ar[r]^-{\beta^\natural}_-\simeq & ((\Chow(k)_\bbQ/\cN)\!/_{\!\!-\otimes \bbQ(1)})^\natural \ar[r]^-{\gamma^\natural}_-\simeq & (\Num(k)_\bbQ \!/_{\!\!-\otimes \bbQ(1)})^\natural \,,
}
\end{equation*}
where the natural equivalence $(\Chow(k)_\bbQ/\cN)^\natural \simeq \Num(k)_\bbQ$ is explained in \cite[Examples 6.3.1 and 7.1.2]{AK}. By choosing inverses $(\beta^\natural)^{-1}$ and $(\gamma^\natural)^{-1}$ to the equivalences $\beta^\natural$ and $\gamma^\natural$, we obtain then a composed functor 
$$\Gamma: \Num(k)_\bbQ\!/_{\!\!-\otimes \bbQ(1)} \to (\Num(k)_\bbQ\!/_{\!\!-\otimes \bbQ(1)})^\natural \stackrel{(\beta^\natural)^{-1} \circ (\gamma^\natural)^{-1}}{\too} ((\Chow(k)_\bbQ\!/_{\!\!-\otimes \bbQ(1)})/\cN)^\natural\,.$$ 
Making use of $\Gamma$ we then construct the following diagram
\begin{equation}\label{eq:aux3}
\xymatrix{
\Chow(k)_\bbQ\!/_{\!\!-\otimes \bbQ(1)} \ar[d]_-{(\gamma \circ \alpha)} \ar[rr]^-R && \NChow(k)_\bbQ \ar[d] \\
\Num(k)_\bbQ\!/_{\!\!-\otimes \bbQ(1)} \ar[d] \ar[d]_{\Gamma}&& \NChow(k)_\bbQ/\cN  \ar[d] \\
((\Chow(k)_\bbQ\!/_{\!\!-\otimes \bbQ(1)})/\cN)^\natural \ar[rr]_-{(R/\cN)^\natural} &&(\NChow(k)_\bbQ/\cN)^\natural=\NNum(k)_\bbQ\,.
}
\end{equation}
Now, note that the commutativity of the above large diagram implies that the left vertical columns of diagrams \eqref{eq:aux3} and \eqref{eq:aux1} are naturally isomorphic. Since the diagram \eqref{eq:aux1} is commutative, we conclude that the above diagram \eqref{eq:aux3} is also commutative (up to natural isomorphism). We define then the functor $R_{\cN}$ to be the composition $(R/\cN)^\natural \circ \Gamma$. Finally, the functor $R_{\cN}$ is fully-faithful, $\bbQ$-linear, additive and symmetric monoidal since each one of the functors used in its construction has these properties.
\section{Proof of Corollary~\ref{cor:applic1}}
Clearly, if $f$ is an isomorphism then $(R\circ \pi)(f)$ and $(R_{\cN}\circ \pi)(f)$ are also isomorphisms. By Theorem~\ref{thm:embedding} the functors $R$ and $R_{\cN}$ are fully-faithful. Hence, it suffices to prove that if by hypothesis $\pi(f)$ is an isomorphism then $f$ is also an isomorphism. This follows from the following general lemma, which is of independent interest.

\begin{lemma}
Let $\cC$ be a symmetric monoidal category and $\cO$ a $\otimes$-invertible object as in \S\ref{sub:orbit}. Then, the functor $\pi: \cC \to \cC\!/_{\!\!-\otimes \cO}$ is conservative.
\end{lemma}
\begin{proof}
Let $f:X \to Y$ be a morphism in $\cC$. We need to show that if by hypothesis $\pi(f):X \to Y$ is an isomorphism in the orbit category $\cC\!/_{\!\!-\otimes\cO}$, then $f$ is an isomorphism in $\cC$. Let 
$$\underline{g} =\{g_i\}_{i \in \bbZ} \in \bigoplus_{i \in \bbZ} \Hom_{\cD}(Y, X \otimes \cO^{\otimes i})$$
be the inverse of $\pi(f)$ in $\cC\!/_{\!\!-\otimes \cO}$. Since $\pi(f)_i=0$ for $i \neq 0$, a direct inspection of the compositions $\underline{g} \circ \pi(f)$ and $\pi(f) \circ \underline{g}$ show us that $g_0 \circ f =f \circ g_0 = \id_X$. As a consequence, $g_0:Y \to X$ is the inverse of $f$ in $\cC$ and so the proof is finished.
\end{proof}
\section{Proof of Corollary~\ref{cor:applic2}}
By construction of the orbit category, we have an injective homomorphism
$$ \End_{\Num(k)_\bbQ}((X,p,m)) \too \End_{\Num(k)_\bbQ\!/_{\!\!-\otimes \bbQ(1)}}(\pi(X,p,m))$$
of $\bbQ$-algebras. Note that Theorem~\ref{thm:embedding} furnish us a $\bbQ$-algebra isomorphism
$$ \End_{\Num(k)_\bbQ\!/_{\!\!-\otimes \bbQ(1)}}(\pi(X,p,m)) \stackrel{\sim}{\too} \End_{\NNum(k)_\bbQ}((R_{\cN} \circ \pi)(X,p,m))$$
and that the proof of Theorem~\ref{thm:semi-simple} guaranties that the $\bbQ$-algebra on the right-hand side is finite dimensional. Hence, by combining the above arguments we conclude that  $\End_{\Num(k)_\bbQ}((X,p,m))$ is finite dimensional.

\medbreak

\noindent\textbf{Acknowledgments:} The authors are very grateful to Aravind Asok for useful discussions concerning the counting number (in the commutative world) and to Alexander Beilinson, Vladimir Drinfeld and Maxim Kontsevich for stimulating questions and conversations. They would also like to thank the Departments of Mathematics of Caltech and MIT for their hospitality and excellent working conditions.

\end{document}